\documentclass[10pt]{amsart}

\usepackage{amsmath,amsthm}
\usepackage[colorlinks=true,linkcolor=blue,urlcolor=blue,citecolor=blue]{hyperref}
\usepackage[all]{xy}

\usepackage{amsrefs} 
\usepackage{enumerate} 
\usepackage{amssymb}  
\usepackage{latexsym} 
\usepackage{comment}
\usepackage{color}
\usepackage{colonequals}
 \usepackage{epsfig}
\usepackage{tikz}
\usepackage{dsfont}

\def\arxiv#1{%
	\href{http://arxiv.org/abs/#1}{arXiv:#1}}%
\def\doi#1{%
	\href{https://doi.org/#1}{doi:#1}}%

\newcommand\sH{\ensuremath{{\mathcal H}}}

\newcommand\sL{\ensuremath{{\mathcal L}}}

\newcommand\sO{\ensuremath{{\mathcal O}}}

\newcommand{\CC}{\ensuremath{\mathbb{C}}}

\newcommand{\NN}{\ensuremath{\mathbb{N}}}
\newcommand{\PP}{\ensuremath{\mathbb{P}}}

\newcommand{\ZZ}{\ensuremath{\mathbb{Z}}}

\DeclareMathOperator{\Aut}{Aut}

\DeclareMathOperator{\Div}{Div}

\DeclareMathOperator{\Pic}{Pic}
\DeclareMathOperator{\PGL}{PGL}

\DeclareMathOperator{\Rank}{Rank}

\begin{document}

\newtheorem{theorem}{Theorem}[section]
\newtheorem{lemma}[theorem]{Lemma}
\newtheorem{proposition}[theorem]{Proposition}
\newtheorem{corollary}[theorem]{Corollary}
\newtheorem*{theorem*}{Theorem}

\theoremstyle{definition}
\newtheorem{remark}[theorem]{Remark}
\newtheorem{definition}[theorem]{Definition}
\newtheorem{example}[theorem]{Example}
\newtheorem{notation}[theorem]{Notation}

\newtheorem*{problem*}{Problem}
\newtheorem*{question*}{Question}
\newtheorem*{acknowledgements}{Acknowledgements}

\numberwithin{equation}{section}

\renewcommand{\subjclassname}{%
	\textup{2020} Mathematics Subject Classification}

\title[Smooth k-double covers ]{Smooth k-double covers of the plane of geometric genus $3$}

\author{Federico Fallucca}
\address{Dipartimento di Matematica\\
	Universit\`a di Trento\\
	via Sommarive 14\\
	38123 Trento\\ Italy.}
\email{Federico.Fallucca@unitn.it}

\author{Roberto Pignatelli}
\address{Dipartimento di Matematica\\
	Universit\`a di Trento\\
	via Sommarive 14\\
	38123 Trento\\ Italy.}
\email{Roberto.Pignatelli@unitn.it}

\begin{abstract}
	In this work we classify all smooth surfaces with geometric genus equal to three and an action of a group $G$ isomorphic to $\left(\mathbb{Z}/2\right)^k$ such that the quotient is a plane.
	We find $11$ families. We compute the canonical map of all of them, finding in particular a family of surfaces with canonical map of degree $16$ that we could not find in the literature.
	We discuss the quotients by all subgroups of $G$ finding several K3 surfaces with symplectic involutions. In particular we show that six families are families of triple K3 burgers in the sense of Laterveer.    
	\end{abstract}

\maketitle

\tableofcontents

\section{Introduction}

The surfaces of general type with geometric genus $3$ are interesting from several different points of view. 
This is not the right place to recall their long story, so we just focus here to some topics that directly involve them where important developments have been seen in the last years through the contribution of various scholars.

A first interest  come from the study of the degree of the canonical map of a surface of general type, see the beautiful recent survey \cite{surveyMLP}, whose degree is bounded from above by $36$, as showed long time ago by Persson in \cite{Ulf}. Recently it has been proved the the bound is sharp in \cites{yeung, LaiYeung, carlos2}. It is still not known if all the integers between $1$ and $36$  can be the degree of the canonical map of some surfaces of general type: it is however known since \cite{beauville} that if the degree is $23$ or more then the surfaces must have $p_g=3$. We refer to the introduction of \cite{fede2} for an account of what we know on this problem, just mentioning here that most of the examples with highest degree of the canonical map that we know are obtained as Galois covers of rational surfaces with Galois group isomorphic to $\left(\mathbb{Z}/2\right)^k$: see for example the examples with canonical map of degree $32$ in \cite{gpr} and those of degree $20$ in \cite{Bin}.

On the other hand, a classical conjecture of Claire Voisin, describing how $0-$cycles on a surface $S$ should behave when pulled-back to a self-product of enough copies of $S$, led Laterveer to the definition of triple K3 burgers. These are surfaces with $p_g=3$  provided with three pairwise commuting involutions such that the quotients are K3 surfaces. 
Studying them, Laterveer proved in \cite{MR4333030} Voisin's conjecture for some families of surfaces, including a family of surfaces with $p_g=3$ (that he calls "Garbagnati surfaces of type G3") with an action of  $\left(\mathbb{Z}/2\right)^2$ whose quotient is $\PP^2$.

This leads us to the problem, interesting by itself, of studying and classifying all surfaces with $p_g=3$ with an action of a group isomorphic to   $G=\left(\mathbb{Z}/2\right)^k$ such that the quotient $S/G$ is isomorphic to $\PP^2$ and then studying their geometry, by investigating their canonical map and if they are triple $K3$ burgers.  We call surfaces like these $k-$double planes for short.

We find that these surfaces form $11$ families. Our classification is summarized by the following Theorem \ref{main}.

\begin{theorem*}
All smooth k-double covers $S$ of the plane with geometric genus $3$ are regular surfaces with ample canonical class. 

They form $11$ unirational families, that we labeled as A1, A2, A3, A4, B2, C3, C4, D3, D4, D5 and E3 in a way that the number in the notation equals $k$. In particular A1 is a family of double planes, A2 and B2 are two families of bidouble planes and so on. 

The degree of the canonical map  $\varphi_{K_S}$ is constant in each family.

We summarize  in the following table the modular dimension (the dimension of its image in the Gieseker moduli space of the surfaces of general type) of each family, and the  values of $K^2_S$ and $\deg \varphi_{K_S}$ of each surface in the family:
\begin{table}[ht!]
\begin{tabular}{|c||ccccccccccc|}
\hline
Family&A1&A2&A3&A4&B2&C3&C4&D3&D4&D5&E3\\
\hline \hline
mod. dim.&36&20&12&8&19&12&8&12&8&6&12\\
$K_S^2$&2&4&8&16&9&8&16&2&4&8&8\\
$\deg \varphi_{K_S}$&2&4&8&16&9&8&16&2&4&8&4\\
\hline
\end{tabular}
\end{table}

In particular the modular dimension of each family  equals $4+2^{6-k}$ with one exception, the family B2, whose dimension is 19. The canonical map is a morphism of degree $K_S^2$ on $\PP^2$ unless $S$ of type $E3$, in which case the canonical map is a rational map of degree $K_S^2-4=4$ undefined at $4$ points. 

\end{theorem*}

Please refer to the beginning of Section \ref{sec: the eleven families} for a detailed explanation of the notation chosen for the 11 families.

The surfaces in the families A1, A2, A3, A4 appeared already in  \cite{DuGao}*{Theorem 1.1}, since their canonical map is exactly the quotient by the action of $G$. 
The surfaces of type B2 are the aforementioned "Garbagnati surfaces of type G3". All the other families seems not to have studied before. 
We note that we have two families with canonical map of degree $16$: the surfaces of type A4 are special cases of a famous construction of Persson, whereas the surfaces of type C4 seems to be new.

We study the geometry of all the quotients of these surfaces by subgroups of $G$, detecting in particular all $K3$ surfaces.
We obtain (Corollary \ref{whichareK3})

\begin{theorem*}
The families B2,  C3,  D3, D4, D5 and E3 are families of triple K3 burgers.
\end{theorem*}

M. Manetti has observed that, reading the definition of triple K3 burger, it naturally raises the question of whether the three involutions generate a group of order 4 or 8.
In all our families of K3 burgers they generate a group of order $4$. We note that the same happens for all  K3 burgers that we know.
In fact, the triple K3 burgers in \cite{Robert} are hypersurfaces in a weighted projective space $\PP(1^3,2s)$ invariant for the group
\[
(x_0,x_1,x_2,y)\rightarrow (\pm x_0,\pm x_1,\pm x_2,y)
\]
that, as a subgroup of the automorphism group of $\PP(1^3,2s)$, has order $4$ (since $2s$ even implies $(x_0,x_1,x_2,y)=(- x_0,- x_1,- x_2,y)$).
The $4$ cases listed in both  \cite{Robert}*{Proposition 3.8 and 3.10} are in fact different linearizations of this subgroup. The three pairwise commuting involutions giving the structure of triple K3 burger are the nontrivial elements of this group.
This raises the following.

{\bf Question:} Is it possible that the three pairwise commuting involutions in the definition of triple $K3$ burger generate a group of order $8$?

We mention again here the family C4, since in this case, exactly as for the triple K3 burgers, there are three involutions in the group $G$ whose quotients have $p_g=1$. The quotient surfaces are not of type K3 but special Horikawa surfaces as those studied in \cite{MR4292206}, which seem to be very similar to K3 surfaces from the point of view of the Voisin's conjecture.

We note that a similar analysis should be possible with the same techniques  for any rational surface, as those considered in Bin's papers as \cite{Bin}. The main argument is that by using the standard formulas for abelian covers, if the Galois group is of the form $(\ZZ/2\ZZ)^k$, the numerical class of all divisors $D_g$ is determined by the characteristic line bundles $L_\chi$. 
We give the explicit formula in Theorem \ref{thm:Dg}. So we first compute the possible $L_\chi$, that is easy, and then deduce from it the class of each divisor. 

This is (unfortunately) not true for general abelian group, since different numerical class of divisors may give the same characteristic sheaves $L_\chi$, see Remark \ref{ExampleZ5}.
So a similar analysis for different groups may be harder. However, there are several interesting examples of Galois covers of rational surfaces with Galois group of the form $\left(\mathbb{Z}/p \ZZ \right)^k$, see for example \cite{fede1} and \cite{BvBP}, so also such a classification would be desirable.

The paper is organized as follows.

In section 1 we recall the general theory of abelian covers and prove the  just mentioned Theorem \ref{thm:Dg} when the group is of the form $(\ZZ/2\ZZ)^k$. In section 2 we recall the known results on the canonical systems of abelian covers. Note that in these two sections we use the multiplicative notation for $G^*$ since it is more efficient for writing the general theory, whereas in the other sections we switch to the additive notation which is more convenient for the computations.

In section $3$ we study and classify all the smooth $k-$double planes, obtaining the $11$ mentioned families in terms of the branch divisors $D_g$ and of the characteristic sheaves $L_\chi$.

In section $4$ we prove Theorem \ref{main}, and then we study each family separately. For each family we write explicit equations in a weighted projective space, and describe the quotients by all subgroups of $G$, determining all the $K3$ surfaces obtained in this way and the symplectic involutions on them.

Finally, in the last section, we determine which families are families of triple K3 burgers.

\subsection*{Notation}
A Galois cover is a finite morphism $\pi\colon X \rightarrow Y$ among algebraic varieties with the property that there is a subgroup $G$ of $\Aut(X)$ such that $\pi$ factors as the composition of the quotient map $X \rightarrow X/G$ with an isomorphism $X/G \cong Y$. We will always assume $Y$ to be irreducible, whereas we find it convenient for the general theory of Galois covers not to do any analogous assumption for $X$. The finite group $G$ is the Galois group of $\pi$. 

An abelian cover is a Galois cover whose Galois group is an abelian group. 
A $k-$double cover is an abelian cover whose Galois group is isomorphic to $ \left( \ZZ/2\ZZ \right)^k$.
A $k-$double plane is a $k-$double cover of $\PP^2$.

\section{Abelian covers}

In this section we collect some preliminary results on abelian covers, mostly well known.

Let  $\pi$ be an abelian cover with $Y$ smooth and $X$ normal. Following  \cite{Rita: paper}, we decompose the direct image of the structure sheaf of $X$ as a sum of line bundles corresponding to the characters of $G$
\[
\pi_* \sO_X = \bigoplus_{\chi \in G^*} \sL_\chi^{-1}.
\]

By the Zariski-Nagata purity theorem, the branch locus of $\pi$ is a divisor. We call this divisor $D$ when it is taken with the reduced structure. 
The ramification divisor $R$ of $\pi$ is the preimage $\pi^{-1}(D)$, also taken with the reduced structure.

Let $T$ be an irreducible component of $R$. By \cite{Rita:paper}*{Lemma 1.1}  the elements of $G$ fixing all points of $T$ form a cyclic subgroup $H$ of $G$, the inertia group of $T$. By \cite{Rita:paper}*{Lemma 1.2},  there is a unique character $\psi \colon H \rightarrow \CC^*$, a generator for the group of characters $H^*$, and a uniformizing parameter $t$ for $\sO_{X,T}$ such that, for all $h \in H$, $h$ acts as 
\[
t \mapsto \psi(h)t.
\] 
This gives a natural decomposition 
\[
R=\sum_{\substack{H < G \text{ cyclic }\\ \psi \text{ generating } H^*}} R_{H,\psi}
\] of the ramification divisor as follows: if $T$ is an irreducible component of $R$, then $T$ is a summand of $R_{H,\psi}$ if and only if its inertia group is $H$ and the corresponding character is $\chi$.

As in \cite{BarbaraRita} we observe that there is a natural bijection among the pairs $(H,\psi)$ as above and the group $G$, associating to each element $g \in G$ the subgroup $H=\langle g \rangle$ generated by it and the unique character $\psi \in H^*$ with the property that  $\psi (g)=e^{\frac{2 \pi i}{|H|}}$. 
So we can set $R_g := R_{H,\psi}$ and write $R=\sum_{g \in G} R_g$. 

Since $G$ is abelian, if $T_1$ and $T_2$ are two irreducible components of $R$ in the same $G-$orbit, they share the same inertia group $H$ and the same character $\psi$, so $T_1$ and $T_2$ belong to the same summand $R_g$.  Therefore there are reduced divisors $D_{g}$ (denoted $D_{H,\psi}$ in \cite{Rita:paper}) such that $R_g=\pi^{-1}(D_g)$. These give a decomposition of the branch divisor
\[
D= \sum_{g \in G} D_g.
\]
 \begin{definition}[\cite{Rita:paper}*{Definition 2.1}] 
 The building data of an abelian cover  $\pi \colon X \rightarrow Y$ are the line bundles $\sL_\chi$ and the reduced effective divisors $D_g$ introduced above.
 \end{definition}
 
 Note that, if $0$ is the identity of $G$, $D_0=0$. Analogously, if $1$ is the trivial character of $G$, $\sL_1\cong \sO_Y$.
 \begin{remark}
 $X$ is connected (equivalently: irreducible) if and only if, for all $\chi \neq 1$, $H^0(\sL_\chi^{-1})=0$. 
\end{remark}
 The building data determine the cover in the following sense. 
 
\begin{definition}
Let $\pi \colon X \rightarrow Y$ be an abelian cover with Galois group $G$, $Y$ smooth and $X$ normal. Fix an element $g \in G$ and a character $\chi \in G^*$. Let $o(g)$ be the order of $g$. Then there exists a unique integer $0 \leq r_g^\chi \leq o(g)-1$ such that 
\[
\chi(g)= e^{r_g^\chi \cdot \frac{2 \pi i} {o(g)}}.
\] 
Given a further character $\chi' \in G^*$ we set moreover
\[
\varepsilon_{\chi,\chi'}^g=
\begin{cases}
1&\text{if }  r_g^\chi +  r^{\chi'}_g \geq o(g)\\
0& \text{else}
\end{cases}.
\]
\end{definition}

\begin{theorem}[\cite{Rita:paper}*{Theorem 2.1 and Corollary 3.1}]\label{thm: rita}
Let $\pi \colon X \rightarrow Y$ be an abelian cover with Galois group $G$, $Y$ smooth and $X$ normal.

Then for all $\chi, \chi' \in G^*$ 
\begin{equation}\label{equation:rita}
\sL_\chi \otimes \sL_{\chi'} \cong \sL_{\chi \cdot \chi'} \otimes \sO_X \left(\sum_{g \in G} \varepsilon_{\chi,\chi'}^g \cdot D_g \right).
\end{equation}
Conversely, given an abelian group $G$ and a smooth irreducible variety $Y$ assume that we have
\begin{itemize}
\item[] a line bundle $\sL_\chi$ on $Y$ for each character $\chi \in G^*$ and
\item[] an effective divisor $D_g$ for all $g \in G$
\end{itemize}
satisfying (\ref{equation:rita}), and with the property that the divisor $D=\sum D_g$ is reduced. 

Then there is a unique Galois cover $\pi \colon X \rightarrow Y$  whose Galois group is $G$, and whose building data are the $L_\chi$ and the $D_g$, such that $X$ is normal.
\end{theorem}

Equation (\ref{equation:rita}) shows that the divisors $D_g$ determine the line bundles $\sL_\chi$ up to torsion as follows.

\begin{definition}
For all $\chi$ set $L_\chi \in \Pic (Y)=\Div(Y)/\sim$ for the divisor class of the invertible sheaf $\sL_\chi$. We use the additive notation for the torsion product in $\Pic (Y)$.
\end{definition}
\begin{corollary}[see \cite{Rita:paper}*{Proposition 2.1}]\label{cor:Lchi}
\begin{equation*}
o(\chi) L_\chi \equiv \sum_{g \in G} \frac{o(\chi)r_g^\chi}{o(g)} D_g.
\end{equation*}
In particular
\begin{equation*}
L_\chi \equiv_{num} \sum_{g \in G} \frac{r_g^\chi}{o(g)} D_g.
\end{equation*}
\end{corollary}
\begin{proof} Note first that by definition of $r_g^\chi$, for all $k \in \NN$, $r_g^{\chi^k}$ is the remainder of the Euclidean division of $kr_g^{\chi}$ by $o(g)$.
Then
\[
\sL_{\chi}^k \cong \sL_{\chi^k} \left( \sum_{g \in G} \left\lfloor \frac{k r_g^\chi}{o(g)} \right\rfloor  D_g \right)  
\]
follows by induction on $k$ applying (\ref{equation:rita}) to the products $\sL_\chi \otimes \sL_{\chi^{k-1}}$.

For $k=o(\chi)$ we obtain the stated formula since $\sL_1 \cong \sO_X$ and $\frac{o(\chi)r_g^\chi}{o(g)}$ is integral.
\end{proof}

In particular, if $\Pic(Y)$ is torsion free (for example if $Y$ is rational) then the divisors do determine the line bundles.

In the next sections we are going to walk in the opposite direction: first we look for the "good" possible $\sL_\chi$ and then we find suitable divisors $D_g$ realizing them. 

Of course the divisors will be free to move in their linear equivalence class. We find it important to notice that for general $G$ the line bundles $\sL_\chi$ do not determine even the linear equivalence class of the divisors $D_g$. In fact this fails already for cyclic groups of order $5$ of more. We just write one example of this phenomenon.

\begin{example}\label{ExampleZ5}
Set $G=\ZZ/5\ZZ=\left\{\overline{0}, \overline{1},\overline{2},\overline{3},\overline{4} \right\}$

Then the following choices
\begin{align*}
\deg D_{\overline{0}}&=0&
\deg D_{\overline{1}}&=2&
\deg D_{\overline{2}}&=0&
\deg D_{\overline{3}}&=0&
\deg D_{\overline{4}}&=2\\
\deg D_{\overline{0}}&=0&
\deg D_{\overline{1}}&=1&
\deg D_{\overline{2}}&=1&
\deg D_{\overline{3}}&=1&
\deg D_{\overline{4}}&=1\\
\end{align*}
give both Galois covers with Galois group $G$  and $\sL_{\chi}\cong \sO_{\mathbb P^1}(2)$ for all $\chi \neq 1$.
\end{example}

In contrast, we show in the forthcoming Theorem \ref{thm:Dg} that when  $G\cong \left( \ZZ/2\ZZ \right)^k$ the $L_\chi$ determine the linear equivalence class of the divisors $D_g$ up to torsion. 

We first need a Lemma on the sums of the $r_g^\chi$ for general abelian covers.

\begin{definition}
The natural isomorphism $G \rightarrow G^{**}$ allows each $g$ in $G$  to be considered as a character of $G^*$, which we will also denote by $g$, by setting 
\[
g(\chi)=\chi(g).
\]
Then $\ker g$ is the subgroup of $G^*$ of the characters $\chi$ such that $\chi(g)=1$. In other words
\[
\chi \in \ker g \Leftrightarrow g \in \ker \chi.
\]
\end{definition}

Let $\sH$ be a subgroup of $G^*$, possibly of the form  $\ker g$.  For all $g \in G$ we will denote by $g_{|\sH}$ the element of $\sH^*$ obtained restricting $g$ to $\sH$.

\begin{lemma}\label{lem: sumrchig}
For all $g\in G$, for each subgroup $\sH$ of $G^*$, 
\begin{equation}\label{eq: sumrchigH}
\sum_{\chi \in \sH} r^\chi_g=\frac{|\sH|}{2} o(g) \left( 1 -\frac{1}{o(g_{|\sH})} \right)
\end{equation}
In particular
\begin{equation}\label{eq: sumrchig}
\sum_{\chi \in G^*} r^\chi_g=\frac{|G|}{2}  \left( o(g)-1 \right).
\end{equation}
\end{lemma}

\begin{proof}
Since $r^\chi_g=0$ if and only if $\chi \in \ker g$, then the number of addenda of $\sum_{\chi \in \sH} r^\chi_g$ that are equal to zero is exactly $
\left| \ker  g_{|\sH}\right|=\frac{|\sH|}{o(g_{|\sH})}$.

The remaining $|\sH| \left( 1 -\frac{1}{o(g_{|\sH})}\right)$ addenda  are integers between $1$ and $o(g)-1$. Since 
$r^\chi_g \neq 0$ implies $ r^\chi_g+r^{\chi^{-1}}_g=o(g)$ it follows that their average equals $\frac{o(g)}2$, thus giving the result.
\end{proof}

It follows that
\begin{proposition}	
\begin{equation}\label{eq: sumLchi}
\sum_{\chi \in G^*}  L_\chi \equiv_{num}\frac{|G|}{2} \sum_{g \in G}  \left( 1 - \frac{1}{o(g)}\right) D_g.
\end{equation}
Moreover, for every $g \in G$,
\begin{equation}\label{eq: sumLchiKerh}
\sum_{\chi \in \ker g}  L_\chi \equiv_{num} \frac{|G|}{2o(g)} \sum_{h \in G}  \left( 1 - \frac{1}{o(h_{|\ker g})}\right) D_h.
\end{equation}
\end{proposition}
\begin{proof}
By Corollary \ref{cor:Lchi}
$
L_\chi \equiv_{num} \sum_{g \in G} \frac{r_g^\chi}{o(g)} D_g
$.

Summing over all $\chi$ and using  \eqref{eq: sumrchig} we obtain \eqref{eq: sumLchi}.

Setting $\sH=\ker g$ and summing only on the characters in $\sH$, using \eqref{eq: sumrchigH} and $|\sH|=\frac{|G|}{o(g)}$ we obtain \eqref{eq: sumLchiKerh}.
\end{proof}

Now we can give the announced formula for the linear systems of the divisors $D_g$ in terms of the $L_\chi$ when the group is $ \left( \ZZ/2\ZZ \right)^k$.

\begin{theorem}\label{thm:Dg}
 Let $\pi \colon X \rightarrow Y$ be a $k-$double cover, $Y$ smooth and $X$ normal, with set of data $\sL_\chi$, $D_g$. Then for all $g \in G$
\[
D_g \equiv_{num} \frac{1}{2^{k-2}} \left( \sum_{\chi \not\in \ker g} L_\chi - \sum_{\chi \in \ker g} L_\chi \right).
\]
\end{theorem}
\begin{proof} Let us fix an element $g \in G= \left( \ZZ/2\ZZ \right)^k$, $g \neq 0$.

We note that for all $h$ in $\left( \ZZ/2\ZZ \right)^k$,  $o(h_{|\ker g})$ equals $1$ if $h \in \langle g \rangle=\{0,g\}$ and $2$ otherwise. Then by (\ref{eq: sumLchiKerh})
\[
\sum_{\chi \in \ker g}  L_\chi \equiv_{num} 2^{k-2} \sum_{h \in G}  \left( 1 - \frac{1}{o(h_{|\ker g})}\right) D_h
= 2^{k-3} \sum_{g \not\in \langle h\rangle}   D_h.
\]
By (\ref{eq: sumLchi}), recalling that $D_0=0$, we obtain  
$\sum_{\chi \in G^*}  L_\chi \equiv_{num}2^{k-2}\sum_{h \in G} D_h$ and
then
\begin{multline*}
D_g= D_g+D_0= \sum_{h \in G} D_h-\sum_{h \not\in \langle g\rangle}   D_h=\frac{1}{2^{k-2}}\left( \sum_{\chi \in G^*} L_\chi -2\sum_{\chi \in \ker g}  L_\chi \right)=\\
=\frac{1}{2^{k-2}} \left( \sum_{\chi \not\in \ker g} L_\chi - \sum_{\chi \in \ker g} L_\chi \right).
\end{multline*}
\end{proof}

\section{The canonical system of an abelian cover}
 
 A canonical divisor $K_X$ on a normal variety $X$ is a Weil divisor, the closure of a canonical divisor of the smooth part of $X$ (see \cite{YPG}*{(1.5)}).
 
 If $ \pi \colon X \rightarrow Y$ is a $G-$cover, then $G$ acts on $\pi_* \left( \sO_{X}(K_X)\right)$ inducing a decomposition on it in eigenspaces
 \[
 \pi_* (\sO_{X}(K_X))=\bigoplus_{\chi \in G^*}  \pi_* (\sO_{X}(K_X))^{(\chi)}
 \]  
 
\begin{theorem}[\cite{rigidity}*{Proposition 2.4}, see also \cite{Rita:paper}*{Proposition 4.1, c) for the case when $X$ is smooth}]\label{thm: canonicalsplitting}  Let  $\pi \colon X \rightarrow Y$ be an abelian cover, with $X$ normal and $Y$ smooth, whose building data are $\sL_\chi$ and  $D_g$. Then
 \begin{equation}\label{eq: canonicalsplitting}
 (\pi_*\sO_X(K_X))^{(\chi)}\cong \sO_{Y}(K_Y) \otimes \sL_{\chi^{-1}}.
 \end{equation}
\end{theorem} 

Consider a global section  $\sigma\in H^0(\sO_{Y}(K_Y) \otimes \sL_{\chi^{-1}})$, and let $(\sigma)\in \Div(Y)$ be the induced effective divisor.
By (\ref{eq: canonicalsplitting}) $\sigma$ determines an element of $H^0(\pi_*\sO_X(K_X)) \cong H^0(\sO_X(K_X))$, whose divisor is, by the proof of  \cite{rigidity}*{Proposition 2.4} (compare also  \cite{Liedtke}*{Section 3.4}),
\begin{equation}\label{eq: pullbackofform} 
\pi^* (\sigma) + \sum_g (o(g)-r_g^{\chi^{-1}}-1)  R_g.
\end{equation}
It follows 
\begin{proposition}\label{prop: base locus}
Assume that all not empty linear systems $\left| K_Y + L_{\chi}\right|$ are base-point-free.

Then the base locus of $|K_X|$ equals
\[
\bigcap_{\substack{\chi \in G^*: \\ |K_Y+L_\chi|\neq \emptyset}}\ \bigcup_{\substack{g \in G :\\ r^\chi_g \neq o(g)-1}}  R_g
\]
\end{proposition}
 
\begin{proof}
Since $\left| K_Y + L_{\chi}\right|$ is base-point-free, if the linear  subsystem of $|K_X|$ corresponding to $H^0(\sO_X(K_X))^{(\chi^{-1})}$ is not empty, by  (\ref{eq: pullbackofform}) its base locus equals
\[
\bigcup_{\substack{g \in G :\\ r^\chi_g \neq o(g)-1}}  R_g.
\]
Since these linear subsystems generate $|K_X|$, its base locus equals their intersection.
\end{proof}

We recall that all linear systems on $\PP^n$ are base-point-free, so Proposition \ref{prop: base locus} gives a complete description of the base locus of the canonical system of any abelian cover of a projective space. For $k-$double covers of $\PP^n$ we obtain as in \cite{gpr}*{Section 2} (see also  \cite{bidouble}*{Section 2}),

\begin{corollary}\label{cor: base locus}
Let $\pi \colon X \rightarrow \PP^n$ be a $k-$double cover with building data $L_\chi$, $D_g$. Then $|K_X|$ is base-point-free if and only if
\[
\bigcap_{ \chi: \deg L_\chi \geq n+1}\ \bigcup_{g \in \ker \chi}  D_g=\emptyset.
\]
 \end{corollary}

\section{Smooth k-double planes with \texorpdfstring{$p_g$}{pg}=3}

\begin{definition}\label{generic}
A smooth $k-$double plane is a $k-$double cover $\pi \colon X \rightarrow \PP^2$  such that  all $D_g$ are smooth, each two of them intersect transversally, and no point in $\PP^2$ belongs to three of them.

 In particular the branch divisor $D=\sum D_g$ is a smooth normal crossing divisor.
\end{definition}

The assumption ensures the smoothness of $X$.
\begin{proposition}
Let $\pi \colon X \rightarrow \PP^2$ be a smooth $k-$double plane. Then $X$ is smooth.
\end{proposition}
\begin{proof}
This is a special case of \cite{Rita:paper}*{Proposition 3.1} (see also \cite{ManettiInventiones}*{Proposition 3.14}).
\end{proof}

\begin{notation}
         It is convenient to consider $G$ and $G^*$ as vector spaces over the field with $2$ elements as in \cite{ManettiInventiones}*{Setup 3.2}.   We are thus going to switch to the additive notation, so for example the sheaf $\sL_1$ will be $\sL_0$ from now on, and for each character $\chi$ we will write $-\chi$ for the character that was called $\chi^{-1}$ in the previous section.
         
         Denote by $e_1, \cdots ,e_k$ the standard basis of $G=\left(\mathbb{Z}/2\right)^k$ and by $\epsilon_1, \ldots, \epsilon_k$ the dual basis of $G^*$.

	To every  smooth $k-$double plane $\pi \colon X \rightarrow \PP^2$ we consider its building data $ L_\chi, D_g$ and the numbers
         \begin{align*}
         d_{g}&:=\deg D_{g},&
         l_{\chi}&:=\deg L_{\chi}.
         \end{align*} 
\end{notation}
Note that $d_0=l_0=0$. 

Note moreover that since $ G=\left( \ZZ/2\ZZ \right)^k$, for each $\chi \in G^*$, $\chi=-\chi$. We will use this often in the following computations.

\begin{definition}\label{def: three cases}
We will say that a smooth $k-$double plane with $p_g=3$ is
\begin{enumerate}
\item[] of type $A$ if $l_{\epsilon_1}=4$, $l_\chi\in\{1,2\}$ for all $\chi \not\in \left\langle \epsilon_1\right\rangle$
\item[] of type $B$ if  $l_{\epsilon_1}=l_{\epsilon_2}=l_{\epsilon_1+\epsilon_2}=3$, $l_\chi\in\{1,2\}$ for all $\chi \not\in \left\langle \epsilon_1,\epsilon_2 \right\rangle$
\item[] of type $C$ if $l_{\epsilon_1}=l_{\epsilon_2}=l_{\epsilon_3}=3$, $l_\chi\in\{1,2\}$ for all $\chi \not\in \{0,\epsilon_1,\epsilon_2,\epsilon_3\}$
\end{enumerate}
\end{definition}

By (\ref{eq: canonicalsplitting}) for a smooth $k-$double plane $\pi \colon X \rightarrow \PP^2$
\begin{equation}\label{eq: pg}
p_g(X)=h^0(\sO_{X}(K_X))=h^0( \pi_* (\sO_{X}(K_X)))=\sum_{\chi \in G^*}  h^0 (\sO_{\PP^2}(l_{\chi}-3)),
\end{equation}
so in all cases of Definition \ref{def: three cases} we obtain $p_g(X)=3$. Conversely 
\begin{proposition}
Up to automorphisms of $G$ every smooth $k-$double plane with $p_g(X)=3$ falls in one of the three cases in Definition \ref{def: three cases}.
\end{proposition}
\begin{proof}
Since $X$ is connected, for all $\chi\neq 0$, $H^0(\sL^{-1}_\chi)=0$ and thus $l_\chi >0$. 

By (\ref{eq: pg}) $l_\chi \leq 4$ and either there is only one $\chi$ with $l_\chi \geq 3$, in which case $l_\chi=4$, or there are three $\chi$ with $l_\chi \geq 3$, all with $l_\chi=3$.

Using an automorphism of $G$, we can reduce the former case to "type A", and the latter case either to "type B" or "type C", depending if the three special characters are linearly dependent or not.
\end{proof}

We  now look at when a $k-$double plane with  $p_g=3$ has canonical system base-point-free.

\begin{lemma}\label{lemma:bpf}
Let $\pi \colon X \rightarrow \PP^2$ be a smooth $k-$double plane with $p_g=3$ of type $t$. Then $|K_X|$ is base-point-free if and only if
\begin{itemize}
\item[] $d_g=0$ for all $g \in \ker \epsilon_1$ when $t=A$;
\item[] $d_g=0$ for all $g \in \ker \epsilon_1 \cap \ker \epsilon_2$ when $t=B$;
\item[] $d_g=0$ for all $g \in \bigcup_{1 \leq i < j\leq 3} \left( \ker \epsilon_i \cap \ker \epsilon_j \right)$ when $t=C$. 
\end{itemize} 
\end{lemma}
\begin{proof}
By Corollary \ref{cor: base locus} $|K_X|$ is base-point-free if and only if
\[
\bigcap_{\chi:l_\chi \geq 3}\ \bigcup_{g \in \ker \chi}  D_g=\emptyset.
\]
For type A we deduce $d_g=0$ for all $g \in \ker \epsilon_1$.

In the remaining cases we have three characters $\chi$ with $l_\chi=3$. We first show that for each $\chi$ with $l_\chi=3$ there is at least one $g \in \ker \chi$ such that $d_g \neq 0$. In fact, in this case $K_{\PP^2}+L_\chi=0$ and thus by \eqref{eq: pullbackofform}  $\sum_{g \in \ker \chi}  R_g$ is a canonical divisor. If $\sum_{g \in \ker \chi}  d_g$ vanished, then this canonical divisor would vanish, and thus $\sO_X(K_X)$ would be isomorphic to $\sO_X$, contradicting $p_g=3$. 

For type C we obtain that $|K_X|$ is base-point-free if and only the following intersection of three divisors
\begin{equation}\label{eq: intersection}
\left( \bigcup_{g \in \ker \epsilon_1}  D_g \right) \cap
\left( \bigcup_{g \in \ker \epsilon_2}  D_g \right) \cap
\left( \bigcup_{g \in \ker \epsilon_3}  D_g \right)
\end{equation}
vanishes, and by our last remark all three divisors are not empty.  Then if there is a $g$ such that $d_g \neq 0$, belonging to two different $\ker \epsilon_i$, then any intersection  point among $D_g$ and one of the $D_h\neq 0$ in the kernel of the third $\epsilon_j$ is in (\ref{eq: intersection}), and thus $|K_X|$ is not base-point-free.

Conversely, if $d_g=0$ for all $g \in \bigcup_{1 \leq i < j\leq 3} \ker \epsilon_i \cap \ker \epsilon_j$ then the three divisors we are intersecting in (\ref{eq: intersection}) have no common components, and thus the intersection is empty since $D$ is a smooth normal crossing divisor.

For type B the result follows similarly using that $\ker \epsilon_1 \cap \ker \epsilon_2 = \ker \epsilon_1  \cap \ker\left( \epsilon_1 + \epsilon_2 \right)= \ker \epsilon_2  \cap \ker\left( \epsilon_1 + \epsilon_2 \right)$ .
\end{proof}

We can now classify the $k-$double planes with $p_g=3$, by considering separately the three cases in Definition \ref{def: three cases}.

For type A we obtain a special case of the situation classified in \cite{DuGao}*{Theorem 1.1}.
\begin{proposition}\label{prop: case 1}
The smooth $k-$double planes with $p_g=3$ of type $A$ form four families, one for each value of $k=1,\ldots, 4$. 

In all cases $\pi$ is the canonical map of $X$, $|K_X|=|\pi^*{\sO_{\PP^2}(1)}|$ is base-point-free and
\begin{align*}
l_0=0&&
l_{\epsilon_1}=4&&
l_\chi=&2 \text{ for all remaining } \chi\\
d_g=0&\text{ for all }g \in \ker \epsilon_1&
&&
d_g=&2^{4-k} \text{ for all }g \not\in \ker \epsilon_1\\
\end{align*}
\end{proposition}
\begin{proof}
By (\ref{equation:rita}), for all $\chi\in G^*$, $l_\chi+l_{\chi+\epsilon_1}=l_{\epsilon_1}+\sum_{g \in G} \varepsilon_{\chi,\chi+\epsilon_1}^g d_g \geq l_{\epsilon_1}=4$.

Since for $\chi$ not in $\langle\epsilon_1 \rangle$ we have $l_\chi \leq 2$, it follows $l_\chi=2$. 

It follows moreover $\sum_{g \in G} \varepsilon_{\chi,\chi+\epsilon_1}^g d_g=0$ so $d_g=0$ for all $g$ that are neither in $\ker \chi$ nor in $\ker \chi +\epsilon_1$. Varying $\chi\in G^*$ this shows that $d_g=0$ for all $g \in \ker \epsilon_1$.

Then by Lemma \ref{lemma:bpf} $|K_X|$ is base-point-free. In fact  $H^0(\sO_X(K_X))$ equals $H^0(\sO_X(K_X))^{(\epsilon_1)}$: this implies that the canonical map is composed with $\pi$. In fact since $\pi_* (\sO_X(K_X))^{(\epsilon_1)} \cong \sO_{\PP^2}(1)$, $\pi$ is exactly the canonical map of $X$ and $|K_X|=|\pi^*{\sO_{\PP^2}(1)}|$. 

Finally by Theorem \ref{thm:Dg}, for all $g \not\in \ker \epsilon_1$,  
\begin{multline*}
d_g= \frac{  \sum_{\chi \not\in \ker g} l_\chi - \sum_{\chi \in \ker g} l_\chi }{2^{k-2}}
= \\
=\frac{\left( 4+(2^{k-1}-1)\cdot 2\right)  - \left( 0+(2^{k-1}-1)\cdot 2\right) }{2^{k-2}}=2^{4-k}. 
\end{multline*}
It follows $k\leq 4$.

We leave to the reader the easy check that all $4$ cases do exist by checking that (\ref{equation:rita}) holds for them.
\end{proof}

To study the next two cases, we preliminarily note that Corollary \ref{cor:Lchi} may be rewritten as $l_\chi=\frac12 \sum_{g \not\in \ker \chi} d_g$ or equivalently
\begin{equation}\label{eq: sumdg}
\forall \chi \in G^*\ \ \ \sum_{g \in \ker \chi} d_g=d-2l_\chi
\end{equation}
where $d:=\deg D=\sum_g d_g$.

For type B we obtain only one family.

\begin{proposition}\label{prop: case 2}
The smooth $k-$double planes of type $B$ with $p_g=3$ form one family, with $k=2$.
These surfaces have a canonical system that is base-point-free and
\begin{align*}
l_0=0&&
l_\chi=3 &\text{ for } \chi \neq 0\\
d_0=0&&
d_g=3 &\text{ for } g \neq 0\\
\end{align*}
\end{proposition}
\begin{proof}
We note that $G$ is the union of the three subgroups $\ker \epsilon_1$, $\ker \epsilon_2$ and $\ker  \left( \epsilon_1 +\epsilon_2 \right)$, which pairwise intersect in  $\ker  \left( \epsilon_1 \cap \epsilon_2 \right)$.
It follows that
\begin{multline*}
\sum_{g \in \ker  \left( \epsilon_1 \cap \epsilon_2 \right)} d_g=\frac12 \left( -d +\sum_{g \in \ker \epsilon_1} d_g +
\sum_{g \in \ker \epsilon_2} d_g +
\sum_{g \in \ker \left( \epsilon_1 +\epsilon_2\right)} d_g \right) \stackrel{\eqref{eq: sumdg}}{=}\\
=d-\left( l_{\epsilon_1} +l_{\epsilon_2} + l_{\epsilon_1 +\epsilon_2}\right)= d-9,
\end{multline*}
so $d \geq 9$.

On the other hand, since $l_{\chi} \leq 2$ for all $\chi \not\in \langle \epsilon_1, \epsilon_2 \rangle$
\[
2^{k-2}d \stackrel{\eqref{eq: sumLchi}}{=}
\sum_{\chi \in G^*} l_	\chi =9+ \sum_{\chi \not\in \langle \epsilon_1, \epsilon_2 \rangle} l_\chi \leq 9+\left( 2^k-4\right) \cdot 2=2^{k+1}+1. 
\]
so $d \leq 8+ \frac1{2^{k-2}}$.

Since by assumption  $k\geq 2$, we conclude that $d=9$ and $k=2$.

The $d_g$ follow by Theorem \ref{thm:Dg}. Since $\epsilon_1 \cap \epsilon_2$ is trivial, Lemma \ref{lemma:bpf} ensures that the canonical system is base-point-free.

We leave to the reader to check that equations (\ref{equation:rita}) are verified.
\end{proof}

For type C we obtain six families. In order to write them clearly we introduce the following rather standard notation.
\begin{notation}
The weight $w(g)$ of an element $g=(g_1,\ldots,g_k) \in \left( \ZZ/2\ZZ \right)^k$ is the number of $g_i$ different from zero. 

For every $h\leq k$ we denote by $w_h(g)$ the number of $g_i$ different from zero with $i \leq h$. 
\end{notation}
In the following we apply this notation to both the elements of $G$ and of $G^*$. 

We note that by  Lemma \ref{lemma:bpf} the canonical system of a k-double plane with $p_g=3$ of type $C$ is base-point-free if and only if $\sum_{g|w_3(g)\leq 1} d_g  = 0$.

Let us set $\epsilon:=\sum_{i=1}^3 \epsilon_i$. We note that $g \in \ker \epsilon$ if and only if $w_3(g)$ is even. It follows that
\[
2\sum_{g|w_3(g)\leq 1} d_g = 3d-\sum_{g} w_3(g)d_g -  \sum_{g|w_3(g)\ even} d_g = 3d-\sum_{i=1}^3 \left( \sum_{g \not\in \ker \epsilon_i} d_g\right) - \sum_{g \in \ker \epsilon} d_g,
\]
from which, by \eqref{eq: sumdg}
\begin{equation}\label{eq: type 3}
\sum_{g|w_3(g)\leq 1} d_g =\frac12 \left( 3d -2\sum l_{\epsilon_i} -d +2l_\epsilon \right)=
d+l_\epsilon-\sum_{i=1}^3 l_{\epsilon_i} =d+l_\epsilon-9.
\end{equation}

We consider first those surfaces whose canonical system is base-point-free.

\begin{proposition}\label{prop: case 3}
The smooth $k-$double planes with $p_g=3$  of type $C$ with canonical system base-point-free form the following five families.
\begin{enumerate}
\item[(C3)] $k=3$, $l_0=0$ and
\begin{align*}
l_{\chi}=3& \text{ if } w(\chi)=1,&
l_{\epsilon}&=1,&
l_\chi=2 &\text{ otherwise};\\
d_g=0& \text{ if } w(g)\leq 1, &
&&
d_g=2 &\text{ otherwise}.\\
\end{align*}
\item[(C4)]  $k=4$, $l_0=0$ and
\begin{align*}
l_{\chi}=3& \text{ if } w(\chi)=1,&
l_{\epsilon}&=1,&
l_\chi=2 &\text{ otherwise};\\
d_g=0& \text{ if } w_3(g)\leq 1, &
&&
d_g=1 &\text{ otherwise}.\\
\end{align*}
\item[(D3)]  $k=3$, $l_0=0$ and
\begin{align*}
l_{\chi}=3& \text{ if } w(\chi)=1,&
l_{\epsilon}&=2,&
l_\chi=1 &\text{ otherwise};\\
d_g=0& \text{ if } w(g)\leq 1, &
d_{e_1+e_2+e_3}&=4,&
d_g=1 &\text{ otherwise}.\\
\end{align*}
\item[(D4)]  $k=4$, $l_0=0$ and
\begin{align*}
l_{\chi}=3& \text{ if } w_3(\chi)=1,&
l_{\chi}=1& \text{ if } w_3(\chi)=w(\chi)=2  &
l_\chi=2 &\text{ otherwise};\\
&&&\text{ or } w_3(\chi)\in \{0,3\}, \chi \not\in \{0,\epsilon\}&&\\
d_g=2& \text{ if } w_3(g)=3,&
d_g=1& \text{ if } w_3(g)=w(g)=2 &
d_g=0 &\text{ otherwise}.\\
\end{align*}
\item[(D5)]  $k=5$, $l_0=0$ and
\begin{align*}
l_{\chi}=3& \text{ if } w_3(\chi)=1,&
l_{\chi}=1& \text{ if } w_3(\chi)=w(\chi)=2  &
l_\chi=2 &\text{ otherwise};\\
&&&\text{ or } w_3(\chi)\in \{0,3\}, \chi \not\in \{0,\epsilon\}&&\\
&&
d_g=1& \text{ if } w_3(g)=w(g)=2&
d_g=0 &\text{ otherwise}.\\
&&& \text{ or } w_3(g)=3,&&
\end{align*}
\end{enumerate}
\end{proposition}

\begin{proof}
Since we are assuming that the canonical system is base-point-free, by Lemma \ref{lemma:bpf}  and \eqref{eq: type 3}
\[
d=9-l_{\epsilon}
\]
and we have to distinguish two cases, depending if $l_\epsilon=1$ or $2$.

We start with the case $l_\epsilon=1$. Then $d=8$.

By (\ref{eq: sumLchi}) $\sum_{\chi \in G^*} l_\chi = 8 \cdot 2^{k-2}=2^{k+1}$ so the average of the $l_\chi$ equals $2$. 
We know the values of five $l_\chi$: $l_0=0$, $l_{\epsilon}=9-8=1$ and the three $l_{\epsilon_i}=3$; their average equals $2$ as well. Since for all remaining $\chi$, $l_\chi \leq 2$ we conclude that they all equal $2$.  By Theorem \ref{thm:Dg} $d_{e_1+e_2}=2^{4-k}$ so $k=3$ or $4$. In both cases we deduce all other $d_g$ by \ref{thm:Dg} obtaining the cases (C3) and (C4) in the statement. 

Otherwise $l_\epsilon=2$ and $d=7$. Then by (\ref{eq: sumdg}) $\sum_{g \in \ker \epsilon_i} d_g=1$, so for each $i=1,2,3$ there exists a unique $g \in \ker \epsilon_i$ such that $d_{g} \neq 0$, that we denote by $g_i$, and $d_{g_i}=1$. 

We show that the $g_i$ are linearly dependent by proving that the vector subspace 
\[V = \bigcap_{i=1}^3 \ker g_i \subset G^* \]
has at most codimension $2$. 

First  we note that if $\chi$ is a character with $l_\chi=1$
different from  $\epsilon_1+\epsilon_2$, $\epsilon_1+\epsilon_3$ and $\epsilon_2+\epsilon_3$, 
then it belongs to $V$. In fact then for all  $i \in \{1,2,3\}$ it holds $l_{\chi+\epsilon_i} \leq 2$ and then  by (\ref{equation:rita})
\[
\sum_{\substack{g \not\in \ker \chi \\g \in \ker \epsilon_i }}d_g =
\sum_{\substack{g \not\in \ker \chi \\g \not\in \ker \left( \chi + \epsilon_i\right) }}d_g =
l_\chi + l_{\chi + \epsilon_i}-l_{\epsilon_i}\leq 1+2-3=0.
\]

Then we note that there are at least two $\chi$ in $V$ with $l_\chi \neq 1$: $0$ and $\epsilon$.
So, setting $A:=\#\left\{\chi \in G^* | l_\chi=1\right\}$, then $\#V \geq A-3+2=A-1$.
On the other hand $A=2\cdot 2^{k}+1-\sum_{\chi \in G^*} l_\chi\stackrel{\eqref{eq: sumLchi}}{=}2^{k+1}+1-7\cdot 2^{k-2}= 2^{k-2}+1$.  
Therefore
\begin{equation}\label{eq: A}
\# V  \geq A -1= 2^{k-2}.
\end{equation}
proving the claim that the $g_i$ are linearly dependent.  

By Lemma \ref{lemma:bpf} $g_i\neq g_j$ when $i\neq j$, so $g_3=g_1+g_2$, and $V$ has exactly codimension $2$, and \eqref{eq: A} is an equality. 
We complete $\epsilon$ to a basis 
$
\epsilon,\epsilon_4,\ldots, \epsilon_k
$
of $V$. Then $\epsilon_1,\ldots,\epsilon_k$ is a basis of $G^*$ respect to which $V=\left\{\chi  | w_3(\chi) \in \{0,3\} \right\}$. Since \eqref{eq: A} is an equality we know exactly which $l_\chi$ are equal to $1$: those in $V$ different from $0$ and $\epsilon$, plus the three characters  $\epsilon_1+\epsilon_2$, $\epsilon_1+\epsilon_3$ and $\epsilon_2+\epsilon_3$.

Note that respect to the basis $e_1,\ldots,e_k$ of $G$ dual to $\epsilon_1,\ldots, \epsilon_k$ we have
\begin{align*}
g_1&=e_2+e_3,&
g_2&=e_1+e_3,&
g_3&=e_1+e_2.&
\end{align*}

Finally we compute all $d_g$ from the $l_\chi$ using Theorem \ref{thm:Dg}. For $g=e_1+e_2+e_3$ we obtain
\[
d_{e_1+e_2+e_3}=\frac{1}{2^{k-2}} \left( \sum_{w_3(\chi) \text{ odd}} l_\chi  - \sum_{w_3(\chi) \text{ even}} l_\chi  \right)
\]
We note  that $l_\chi$ appears in this expression with the opposite sign of  $l_{\chi +\epsilon}$. 

Since $w_3(\chi)=3-w_3(\chi+\epsilon)$, then $\chi \in V=\left\{\chi  | w_3(\chi) \in \{0,3\} \right\}$ if and only if $\chi+\epsilon \in V$.
We have proved that, if $\chi$ does not belong to $\left\langle \epsilon_1,\epsilon_2,\epsilon_3\right\rangle$ then $l_\chi=1$ if $\chi \in V$ and $l_\chi=2$ otherwise. So the contributions of the $l_\chi$ not in $\left\langle \epsilon_1,\epsilon_2,\epsilon_3\right\rangle$ cancel each other out and
\[
d_{e_1+e_2+e_3}=\frac{1}{2^{k-2}} \left( l_{\epsilon_1}+ l_{\epsilon_2}+ l_{\epsilon_3}+ l_{\epsilon}- l_{\epsilon_1+\epsilon_2}- l_{\epsilon_1+\epsilon_3}- l_{\epsilon_2+\epsilon_3}\right) =\frac{8}{2^{k-2}}=2^{5-k}\]
so $k\leq 5$ and we obtain a family for each $k=3,4,5$. We leave to the reader the computation of the remaining $d_g$, giving the families $(D3)$, $(D4)$ and $(D5)$.
\end{proof}

Finally we consider those $k-$double planes with $p_g=3$ whose canonical system is not base-point-free, and see that they provide exactly one more family.
\begin{proposition}\label{prop: case bp}
The smooth $k-$double planes with $p_g=3$  whose canonical system is not base-point-free are of type C and form one family, with $k=3$, $l_0=0$ and
\begin{align*}
l_{\epsilon_i}=3&&
l_{\epsilon_1+\epsilon_2}=1&&
l_\chi=2 \text{ otherwise}&&
&\\
d_{e_1+e_2+e_3}=3&&
d_{e_1+e_2}=2&&
d_{e_3}=d_{e_1+e_3}=d_{e_2+e_3}=1 &&
d_0=d_{e_1}=d_{e_2}=0\\
\end{align*}
Their canonical system has four simple base points, the preimages of the two points in the intersection of the line $D_{e_3}$ and the conic $D_{e_1+e_2}$.
\end{proposition}
\begin{proof}
By propositions \ref{prop: case 1} and \ref{prop: case 2} these double planes are of type $C$. Thus, using \eqref{eq: sumLchi}
\begin{equation}\label{eq: fail by 1}
d \cdot 2^{k-2}=\sum_{\chi \in G^*} l_\chi =9 + \sum_{\chi \not\in \{\epsilon_j\}} l_\chi \leq 2^{k+1}+1
\end{equation}
from which we deduce, since $k\geq 3$, $d \leq 8+ \frac1{2^{k-2}} \leq 8+\frac12$. So $d \leq 8$.

We recall that the existence of base points for the canonical system is equivalent to $\sum_{g|w_3(g)\leq 1} d_g  \neq 0$.
On the other hand by \eqref{eq: type 3}
\[
\sum_{g|w_3(g)\leq 1} d_g  =d+l_\epsilon-9 \leq l_{\epsilon}-1.
\]
We conclude that
\begin{align*}
l_\epsilon&=2&
d&=8&
\sum_{g|w_3(g)\leq 1} d_g&=1
\end{align*}
and thus there is an unique $h \in G$ with $d_h=1$ and $w_3(h)=1$. 
Note that exactly one of the three characters $\epsilon_j$ is not in $\ker h$. 

The inequality in \eqref{eq: fail by 1} fails to be an equality exactly by $1$. This means that there is exactly one character $\eta$ with $l_\eta=1$.
By the expression of $d_h$ in term of the $l_\chi$ in Theorem \ref{thm:Dg} we deduce that $\eta \not\in \ker h$ (or $d_h$ would be negative) and $d_h=\frac{1}{2^{k-3}}$. So $k=3$.

Using an automorphism of $G$ we can now assume without loss of generality $\eta=\epsilon_1+\epsilon_2$. We have now computed all $l_\epsilon$: we leave to the reader to compute all $d_g$ by applying Theorem \ref{thm:Dg}.

By \eqref{eq: pullbackofform} the canonical system $|K_X|$ is generated by the following three divisors
\begin{align*}
R_{e_3}&+R_{e_2+e_3}&
R_{e_3}&+R_{e_1+e_3}&
&R_{e_1+e_2}
\end{align*}
and then by the smoothness assumption the base locus is the schematic intersection $ R_{e_1+e_2} \cap R_{e_3}$. 

The line $D_{e_3}$ and the conic $D_{e_1+e_2}$ intersect transversally in two points. Above each of them there are two points of $X$, stabilized by the index two subgroup $\left\langle e_1+e_2,e_3 \right\rangle$, the intersection points of $ R_{e_1+e_2} \cap R_{e_3}$. A straightforward local computation shows that $ R_{e_1+e_2}$ and $R_{e_3}$ are transversal.
\end{proof}

\section{The eleven families}\label{sec: the eleven families}

In the previous section we have proved that the smooth k-double planes with $p_g=3$ form $11$ families. In this section we will study these families. 

\begin{paragraph}{\bf Notation}
We will denote each family by a letter and a number. The number is the exponent $k$ of the Galois group, while the letter reminds the type. 
In particular the $4$ families in Proposition \ref{prop: case 1} give surfaces of type A1, A2, A3 and A4, while the surfaces in Proposition \ref{prop: case 2} form the family B2. There are more families of surfaces of type C with the same Galois group, so for these we need to use more letters: we will use the letters C, D and E. 
Precisely the surfaces in Proposition \ref{prop: case 3} are named, as already specified in that statement, as C3, C4, D3, D4 and D5, while the surfaces in Proposition \ref{prop: case bp} form the family E3.
\end{paragraph}

All these surfaces have ample canonical class, since it is numerically the pull-back of an ample class $\PP^2$ (see {\it e.g.} \cite{Rita:paper}*{Proof of Proposition 4.2}). Their irregularity vanishes, for example since their geometric genus is $3$ by construction and the Euler characteristic is $4$ by \cite{Rita:paper}*{(4.8)}. 

For each family we compute the degree of the canonical map.

\begin{theorem}\label{main}
All smooth k-double covers $S$ of the plane with geometric genus $3$ are regular surfaces with ample canonical class. 

They form $11$ unirational families, that we labeled as A1, A2, A3, A4, B2, C3, C4, D3, D4, D5 and E3 in a way that the number in the notation equals $k$. In particular A1 is a family of double planes, A2 and B2 are two families of bidouble planes and so on. 

The degree of the canonical map  $\varphi_{K_S}$ is constant in each family.

We summarize  in the following table the modular dimension (the dimension of its image in the Gieseker moduli space of the surfaces of general type) of each family, and the  values of $K^2_S$ and $\deg \varphi_{K_S}$ of each surface in the family:
\begin{table}[ht!]
\begin{tabular}{|c||ccccccccccc|}
\hline
Family&A1&A2&A3&A4&B2&C3&C4&D3&D4&D5&E3\\
\hline \hline
mod. dim.&36&20&12&8&19&12&8&12&8&6&12\\
$K_S^2$&2&4&8&16&9&8&16&2&4&8&8\\
$\deg \varphi_{K_S}$&2&4&8&16&9&8&16&2&4&8&4\\
\hline
\end{tabular}
\end{table}

In particular the modular dimension of each family  equals $4+2^{6-k}$ with one exception, the family B2, whose dimension is 19. The canonical map is a morphism of degree $K_S^2$ on $\PP^2$ unless $S$ of type $E3$, in which case the canonical map is a rational map of degree $K_S^2-4=4$ undefined at $4$ points. 
\end{theorem}

\begin{proof}
Each surface $S$ is a Galois cover $\pi \colon S \rightarrow \PP^2$. By the Leray spectral sequence, $H^1(S, \sO_S) \cong H^1(\PP^2, \pi_* \sO_S)\cong \bigoplus_\chi H^1\left( \PP^2, \sL_\chi^{-1}\right)$. Since every line bundle on $\PP^2$ has trivial first cohomology group, it follows $h^1(S, \sO_S)=0$.

The value of the self-intersection of the canonical class follows by the formula (see \cite{Rita:paper}*{(4.8)})
\[
K^2=2^k\left(-3+\frac12 \sum_{g \in G} d_g \right)^2
\]
By Propositions \ref{prop: case 1}, \ref{prop: case 2}, \ref{prop: case 3}, \ref{prop: case bp} the canonical system of $S$ is base point free unless $S$ is of type E3, in which case it has four simple base points. So (blowing up the base points in this last case) we get a surface with canonical system having movable part of self intersection  as in the second line of the table above, so strictly positive. Then the canonical map is not composed with a pencil. Since $p_g=3$ then the canonical map of this surface is a morphism on $\PP^2$ of the given degree.

The families are parametrized by a Zariski open subset of a product of projective spaces, the complete linear systems to which the divisors $|D_g|$, quoted by the faithful action of $\PGL(3)$, a group of dimension $8$. Since the surfaces are of general type, their automorphism group is finite and therefore it contains only finitely many subgroups of the form $(\ZZ/2\ZZ)^k$, which implies that the map from this quotient to the Gieseker moduli space of the surfaces of general type is finite. So the modular dimension of each family equals
\[
-8+\sum \dim |D_g| 
\]  
which gives the modular dimensions in the table above.
As an example, the family $E3$ depends on the choice of three lines, a conic and a cubic so its modular dimension is
\[
-8 + 3 \cdot 2+ 5 +9=12.
\]
\end{proof}

For each family we will first give explicit equations of the surfaces embedded in a suitable weighted projective space, computed  by using the equations in  \cite{Cetraro}*{Section 6} (see also \cite{ManettiInventiones}*{Section 3.3}) as follows. 

We consider a weighted projective space whose first three variables  $x_0,x_1,x_2$ of weight $1$. 
The group acts trivially on them: in fact the $k$-double cover is the map on $\PP^2$ given by them.
Each branch divisor divisor $D_g$,  $g=\sum_1^k i_je_j$, is defined by a homogeneous polynomial in the $x_j$, the polynomial $f_{i_1\cdots i_k}(x_j) \in \CC[x_0,x_1,x_2]$. If $D_g=0$ then $f_{i_1\cdots i_k}(x_j)=1$.

Then we add variables  $y_{i_1\cdots i_k}$, $i_j \in \{ 0,1 \}$, meaning that $e_j$ acts on $y_{i_1\cdots i_k}$ via multiplication by $(-1)^{i_j}$. The equations
\begin{align}\label{eq: equations}
y_{r_1\cdots r_k} y_{s_1\cdots s_k}= y_{t_1\cdots t_k} \prod_{\substack{\sum i_j r_j, \sum i_j s_j\\ \text{ both odd} }} f_{i_1\cdots i_k} &&&
\text{when all } r_j+s_j+t_j \text{ are even}
\end{align}
define an embedding of these surfaces in the weighted projective with variables $x_j, y_{i_1\cdots i_k}$. The weight of the variable  $y_{i_1\cdots i_k}$ is  the positive integer $l_{\sum_j i_j\epsilon_j}$.

Sometimes these equations allow to eliminate some variables, embedding the surfaces in a weighted projective space of smaller dimension. For example for the family A2 we find the equation $y_{11}y_{01}=y_{10}$, using it to eliminate $y_{10}$ gives an embedding in a smaller dimensional weighted projective space. In the following we will eliminate all the variables that we can eliminate, to give simpler equations.

Then we will discuss all "intermediate" quotients, the quotients of these surfaces by subgroups of the Galois group of the cover, with a focus on $K3$ surfaces and symplectic involutions. 

\newpage

\subsection{Family A1}
These surfaces have $K^2=2$. 

They are the hypersurfaces of degree $8$ in $\PP(1^3,4)$, with variables $x_0,x_1,x_2,y_1$,
\[y_1^2=f_1(x_j),
\]
with $\deg f_1=8$.
 
These are the Horikawa surfaces in \cite{Hor1}[Theorem 1.6.(i)].

\subsection{Family A2}
These surfaces have $K^2=4$. 

These are the complete intersections of two quartics in $\PP(1^3,2^2)$, with variables $x_0,x_1,x_2,y_{11},y_{01}$,
\begin{align*}
y_{11}^2&=f_{10}(x_j)&
y_{01}^2&=f_{11}(x_j)&
\end{align*}
with $\deg f_\bullet=4$.

There are three intermediate quotients: the quotient by $e_1$ and $e_1+e_2$ are double planes branched on quartics, so del Pezzo surfaces of degree $2$.
The quotient by $e_2$ is a double plane branched on both quartics, so a degeneration of the family A1, a Horikawa surface with $16$ nodes.

This family is in \cite{DuGao}*{Theorem 1.1.(5)}. These surfaces were also studied by Horikawa, see \cite{Hor4}*{Theorem 2.1}.

\subsection{Family B2} 
These surfaces have $K^2=9$. 

They are embedded in $\PP(1^3,3^3)$, with variables $x_0,x_1,x_2,y_{10},y_{01},y_{11}$, defined by the equations
\[\Rank
\begin{pmatrix}
f_{10}(x_j)&y_{10}&y_{11}\\
y_{10}&f_{11}(x_j)&y_{01}\\
y_{11}&y_{01}&f_{01}(x_j)\\
\end{pmatrix}=1
\]
with $\deg f_\bullet=3$.

The three intermediate quotients are double planes branched on the union of two cubics: three $K3$ surfaces with $9$ nodes.

We met this family in \cite{bidouble}*{Example 6} and \cite{Alice}*{Proposition 6.3}. They are also studied in \cite{MR4333030} and \cite{AliceMatteo}.

\newpage 

\subsection{Family A3}
These surfaces have $K^2=8$. 

They are embedded in $\PP(1^3,2^6)$, with variables $x_0$, $x_1$, $x_2$, $y_{010}$, $y_{001}$, $y_{110}$, $y_{101}$, $y_{011}$, $y_{111}$,  defined by the equations
\[\Rank
\begin{pmatrix}
f_{111}(x_j)&y_{010}&y_{001}&y_{111}\\
y_{010}&f_{110}(x_j)&y_{011}&y_{101}\\
y_{001}&y_{011}&f_{101}(x_j)&y_{110}\\
y_{111}&y_{101}&y_{110}&f_{100}(x_j)\\
\end{pmatrix}=1
\]
with $\deg f_\bullet=2$.

The quotients by $\ker \epsilon_1$ are double planes branched on the union of $4$ conics, degenerations of the family A1 with $24$ nodes. The quotients by each of the other $6$ subgroups of index $2$ are double planes branched on the union of $2$ conics, del Pezzo surfaces of degree $2$ with $4$ nodes.

The quotients by a subgroup $\langle g \rangle$ of index $4$ behave differently according to if $g$ belongs to $\ker \epsilon_1$ or not.
If $g \in \ker \epsilon_1$ the quotient is a degeneration of the family A2 with $16$ nodes. Otherwise, for the remaining four $g$, the quotients are 2-double planes such that each of the three branching divisors is a conic. By, {\it e.g} \cite{rigidity}*{Propositions 2.4-2.5 and their proof} they have $p_g=0$ and bicanonical sheaf trivial, so they are Enriques surfaces.

These surfaces are in \cite{DuGao}*{Theorem 1.1.(3)}, where the authors give them through equations of a different (not normal) birational model.

\newpage

\subsection{Family C3}
These surfaces have $K^2=8$. 

They are embedded in $\PP(1^4,2^3)$, with variables $x_0,x_1,x_2,y_{111},y_{110},y_{101},y_{011}$, defined by the equations
\begin{align*}
\Rank
\begin{pmatrix}
f_{110}(x_j)&y_{011}&y_{101}\\
y_{011}&f_{101}(x_j)&y_{110}\\
y_{101}&y_{110}&f_{011}(x_j)\\
\end{pmatrix}&=1&
y_{111}^2=f_{111}(x_j)
\end{align*}
with $\deg f_\bullet=2$. 

The Galois group has seven subgroups of index $2$, the three of the form $\ker \epsilon_i$, the three of the form $\ker \epsilon_i+\epsilon_j$, and $\ker \epsilon$.

The quotients by a subgroup of the form $\ker \epsilon_i$ are double planes branched on the union of $3$ conics, so K3 surfaces with $12$ nodes. 
The quotients by a subgroup of the form $\ker \epsilon_i +  \epsilon_j$ are double planes branched on the union of $2$ conics, so del Pezzo surfaces of degree $2$ with $4$ nodes. The quotients by  $\ker \epsilon$ are double planes branched on one conic, so $\PP^1 \times \PP^1$. 

The quotients by a subgroup $\langle g \rangle$ of index $4$ are $2-$double planes as follows.
If $g=e_1+e_2+e_3$ then the three branching divisors are three smooth conics, so the quotients are smooth Enriques surfaces. If $g$ is of the form $e_i+e_j$ then one of the branching divisors is empty, one is a smooth conic, and the last is union of two conics: the quotients are K3 surfaces with $8$ nodes.
If $g$ is one of the $e_i$ then two divisors are conics whereas the third is the union of two conics: they are surfaces with $K^2=4$, $p_g=2$ and $8$ nodes.

Then each surface in this family dominates six different K3 surfaces. Let us give names to them. Let $U_{i,j}$ be the K3 with $8$ nodes obtained quoting by $\langle e_i+e_j \rangle$ and let $V_{k}$ be the K3 with $12$ nodes obtained quoting by  $\ker \epsilon_k$. 
Then these K3 are naturally subdivided in three pairs by double covers  $V_{i,j} \to U_k$ (here $k\not\in\{i,j\}$) branched on $8$ nodes and nowhere else, quotient of  $V_{i,j}$ by the symplectic involution induced by $e_i$.
The $V_{i,j}$ are special cases of the K3 surfaces considered in \cite{vGS}*{3.5}, where the plane quartic considered there splits as union of two conics.

\newpage

\subsection{Family D3}
These surfaces have $K^2=2$. 

They are embedded in $\PP(1^6,2)$ with variables $x_0,x_1,x_2,y_{110},y_{101},y_{011},y_{111}$,  defined by the equations
\begin{align*}
\Rank
\begin{pmatrix}
f_{110}(x_j)&y_{011}&y_{101}\\
y_{011}&f_{101}(x_j)&y_{110}\\
y_{101}&y_{110}&f_{011}(x_j)\\
\end{pmatrix}&=1&
y_{111}^2=f_{111}(x_j)
\end{align*}
with $\deg f_{110}=\deg f_{101}=\deg f_{011}=1$ and $\deg f_{111}=4$. Note that these equations are identical to those of the family C3, the only difference being in the degrees.

The quotients by a subgroup of the form $\ker \epsilon_i$ are double planes branched on the union of $2$ lines and one quartic, so K3 surfaces  with $9$ nodes. 
The quotients by a subgroup of the form $\ker \epsilon_i +  \epsilon_j$ are double planes branched on the union of $2$ lines, so del Pezzo surfaces of degree $8$ with $1$ node. The quotients by  $\ker \epsilon$ are double planes branched on one quartic, so smooth del Pezzo surfaces of degree $2$. 

The quotients by a subgroup $\langle g \rangle$ of index $4$ are $2-$double planes as follows.
If $g=e_1+e_2+e_3$ then the three branching divisors are lines, so the quotients are projective planes $\PP^2$. If $g$ is of the form $e_i+e_j$ then one of the branching divisors is empty, one is a smooth quartic, and the last is union of two lines: the quotients are K3 surfaces with $2$ nodes.
If $g$ is one of the $e_i$ then two divisors are lines whereas the third is the union of a line and a quartic: they are surfaces with $K^2=1$, $p_g=2$ and $8$ nodes.

Then each surface in this family dominates six different K3 surfaces naturally subdivided in three pairs as in the previous case. More precisely, let $U_{i,j}$ be the K3 with $2$ nodes obtained quoting by $\langle e_i+e_j \rangle$ and let $V_{k}$ be the K3 with $9$ nodes obtained quoting by  $\ker \epsilon_k$. 
Then we have double covers  $V_{i,j} \to U_k$, $k\not\in\{i,j\}$, branched on $8$ nodes and nowhere else, quotient of  $V_{i,j}$ by the symplectic involution induced by $e_i$. These are again special cases of the K3 surfaces considered in \cite{vGS}*{3.5}, where the plane conic considered there splits as union of two lines.

We finally note that, since the quotient by $e_1+e_2+e_3$ represents these surfaces as double cover of the plane, these surfaces are a degeneration of the surfaces in the family A1, special Horikawa surfaces in the family of \cite{Hor1}*{Theorem 1.6.(i)} with extra automorphisms.

\newpage

\subsection{Family E3}
These surfaces have $K^2=8$. 

They are embedded in $\PP(1^4,2^3,3^2)$, with variables $x_0,$ $x_1,$ $x_2,$ $y_{110},$ $y_{101},$ $y_{011},$ $y_{111},$ $y_{100},$ $y_{010}$, defined by the equations
\[\Rank
\begin{pmatrix}
	f_{110}f_{111}&y_{100}&y_{010}\\
	y_{100}&f_{101}&y_{110}\\
	y_{010}&y_{110}&f_{011}\\
\end{pmatrix}=1, \qquad \Rank\begin{pmatrix}
f_{110}f_{001}&y_{011}&y_{101}\\
y_{011}&f_{101}&y_{110}\\
y_{101}&y_{110}&f_{011}\\
\end{pmatrix}=1,
\]
\[
\Rank\begin{pmatrix}
	f_{111}&y_{100}&y_{111}\\
	y_{100}&f_{110}f_{101}&y_{011}\\
	y_{111}&y_{011}&f_{001}\\
\end{pmatrix}=1, \qquad \Rank\begin{pmatrix}
f_{111}&y_{010}&y_{111}\\
y_{010}&f_{110}f_{011}&y_{101}\\
y_{111}&y_{101}&f_{001}\\
\end{pmatrix}=1.
\]
with $\deg f_{101}=\deg f_{011}=\deg f_{001}=1$, $\deg f_{110}=2$ and $\deg f_{111}=3$.

The quotients by the subgroup $\ker \epsilon_1$  or $\ker \epsilon_2$ are double planes branched on the union of one line, one conic and one cubic, so K3 surfaces  with $11$ nodes.
The quotients by the subgroup $\ker \epsilon_3$ are branched on the union of three lines and one cubic, so K3 surfaces  with $12$ nodes.
The quotients by  $\ker \epsilon_1 +  \epsilon_2$ are branched on the union of two lines, so del Pezzo surfaces  of degree $8$ with $1$ node.
The quotients by  $\ker \epsilon_1 +  \epsilon_3$ or  $\ker \epsilon_2 +  \epsilon_3$ are branched on the union of two lines and a conic, so del Pezzo surfaces of degree $2$  with $5$ nodes.
The quotients by $\ker \epsilon$
are branched on the union of one line and one cubic, so del Pezzo surfaces of degree $2$ with $3$ nodes. 

The quotients by a subgroup $\langle g \rangle$ of index $4$ are $2-$double planes as follows.
If $g=e_1+e_2+e_3$ then two of the branching divisors are lines and the third is the union of a line and a conic, so the quotients are del Pezzo surfaces of degree $1$ with $4$ nodes. If $g=e_1+e_2$ then one divisor is empty, the second is the union of two lines,  the third is the union of a line and a cubic, and the quotients are K3 surfaces with $8$ nodes. If $g$ is $e_1+e_3$ or  $e_2+e_3$ then one of the branching divisors is a line, one is a cubic, and the last is union of a line and a conic: the quotients have $K^2=p_g=1$ and $4$ nodes.
If $g=e_3$ then two divisors are lines and the third is union of a conic and a cubic: the quotients have $K^2=1$, $p_g=2$ and $12$ nodes.
If $g$ is $e_1$ or $e_2$ then one divisor is the union of two lines, one is a conic and the last is the union of a line and a cubic, giving surfaces with $K^2=4$, $p_g=2$ and $8$ nodes.

Then each surface in this family dominates four different K3 surfaces. We get only one symplectic involution by the construction, on the K3 surface with $8$ nodes quotient by $e_1+e_2$. The symplectic involution is induced by $e_1$, and the quotient is the K3 with $12$ nodes obtained by $\ker \epsilon_3$. The two K3 surfaces with $11$ nodes are both dominated by a surface of general type with $K^2=p_g=1$.

\newpage

\subsection{Family A4} 
These surfaces have $K^2=16$. 

They are embedded in $\PP(1^3,2^{14})$ with variables $x_0,$ $x_1,$ $x_2,$ $y_{1111},$ $y_{0100},$ $y_{0010},$ $y_{0001},$ $y_{1011},$ $y_{1101},$ $y_{1110},$ $y_{0110},$ $y_{0101},$ $y_{0011},$ $y_{1010},$ $y_{1100},$ $y_{0111},$ $y_{1001}$, defined by the equations
\[\Rank
\begin{pmatrix}
f_{1000}f_{1011}&y_{1111}&y_{1011}&y_{1100}\\
y_{1111}&f_{1110}f_{1101}&y_{0100}&y_{0011}\\
y_{1011}&y_{0100}&f_{1100}f_{1111}&y_{0111}\\
y_{1100}&y_{0011}&y_{0111}&f_{1010}f_{1001}\\
\end{pmatrix}=1
\]
\[\Rank
\begin{pmatrix}
f_{1000}f_{1101}&y_{1111}&y_{1101}&y_{1010}\\
y_{1111}&f_{1110}f_{1011}&y_{0010}&y_{0101}\\
y_{1101}&y_{0010}&f_{1010}f_{1111}&y_{0111}\\
y_{1010}&y_{0101}&y_{0111}&f_{1100}f_{1001}\\
\end{pmatrix}=1
\]

\[\Rank
\begin{pmatrix}
f_{1000}f_{1110}&y_{1111}&y_{1110}&y_{1001}\\
y_{1111}&f_{1101}f_{1011}&y_{0001}&y_{0110}\\
y_{1110}&y_{0001}&f_{1001}f_{1111}&y_{0111}\\
y_{1001}&y_{0110}&y_{0111}&f_{1100}f_{1010}\\
\end{pmatrix}=1
\]

\[\Rank
\begin{pmatrix}
f_{1100}f_{1101}&y_{0100}&y_{0110}&y_{1010}\\
y_{0100}&f_{1110}f_{1111}&y_{0010}&y_{1110}\\
y_{0110}&y_{0010}&f_{1010}f_{1011}&y_{1100}\\
y_{1010}&y_{1110}&y_{1100}&f_{1000}f_{1001}\\
\end{pmatrix}=1
\]

\[\Rank
\begin{pmatrix}
f_{1100}f_{1110}&y_{0100}&y_{0101}&y_{1001}\\
y_{0100}&f_{1101}f_{1111}&y_{0001}&y_{1101}\\
y_{0101}&y_{0001}&f_{1001}f_{1011}&y_{1100}\\
y_{1001}&y_{1101}&y_{1100}&f_{1000}f_{1010}\\
\end{pmatrix}=1
\]

\[\Rank
\begin{pmatrix}
f_{1010}f_{1110}&y_{0010}&y_{0011}&y_{1001}\\
y_{0010}&f_{1011}f_{1111}&y_{0001}&y_{1011}\\
y_{0011}&y_{0001}&f_{1001}f_{1101}&y_{1010}\\
y_{1001}&y_{1011}&y_{1010}&f_{1000}f_{1100}\\
\end{pmatrix}=1
\]

\[\Rank
\begin{pmatrix}
f_{1000}f_{1111}&y_{1011}&y_{1101}&y_{1110}\\
y_{1011}&f_{1100}f_{1011}&y_{0110}&y_{0101}\\
y_{1101}&y_{0110}&f_{1010}f_{1101}&y_{0011}\\
y_{1110}&y_{0101}&y_{0011}&f_{1001}f_{1110}\\
\end{pmatrix}=1
\]
with $\deg f_\bullet=1$.

The quotients by $\ker \epsilon_1$ are double planes branched on the union of $8$ lines, degenerations of the family A1 with $28$ nodes. The quotients by each of the other $6$ subgroups of index $2$ are double planes branched on the union of $4$ lines. They are del Pezzo surfaces of degree $2$ with $6$ nodes.

The quotients by a subgroup $H$ of index $4$ behave differently according to if $H$ is contained in $\ker \epsilon_1$ or not.
If $H \subset \ker \epsilon_1$ the quotients are degenerations of the family A2 with $24$ nodes. Otherwise, the quotients are Enriques surfaces with $6$ nodes.

The quotients by a subgroup $\langle g \rangle$ of index $8$ also behave differently according to if $g$ belongs to $\ker \epsilon_1$ or not.
If $g \in \ker \epsilon_1$ the quotients are degenerations of the family A3 with $32$ nodes. Otherwise the quotients are numerical Campedelli surfaces, surfaces with $p_g=0$, $K^2=2$ and ample canonical class.

Note that these surfaces are then double covers of numerical Campedelli surfaces: in fact they were first found by Persson in this way in \cite{Ulf}*{Ex. 5.8}. 
They are also in \cite{DuGao}*{Theorem 1.1.(1)}, where the authors give them through equations of a different (not normal) birational model.

\newpage

\subsection{Family C4}
These surfaces have $K^2=16$. 

They are embedded in $\PP(1^4,2^{11})$ with variables $x_0,$ $x_1,$ $x_2,$ $y_{1110},$ $y_{1001},$ $y_{0011},$ $y_{0101},$ $y_{1111}$ $y_{0111},$ $y_{1101},$ $y_{1011},$ $y_{0001},$ $y_{1100},$ $y_{1010},$ $y_{0110}$  defined by the equations
\begin{align*} 
\Rank \left(
\begin{smallmatrix}
f_{1110}&y_{1001}&y_{0101}&y_{0011}&y_{1111}&y_{1110}\\
y_{1001}&f_{1010}f_{1100}f_{0111}&y_{1100}f_{1100}&y_{1010}f_{1010}&y_{0110}f_{0111}&y_{0111}\\
y_{0101}&y_{1100}f_{1100}&f_{1100}f_{0110}f_{1011}&y_{0110}f_{0110}&y_{1010}f_{1011}&y_{1011}\\
y_{0011}&y_{1010}f_{1010}&y_{0110}f_{0110}&f_{1010}f_{0110}f_{1101}&y_{1100}f_{1101}&y_{1101}\\
y_{1111}&y_{0110}f_{0111}&y_{1010}f_{1011}&y_{1100}f_{1101}&f_{1101}f_{1011}f_{0111}&y_{0001}\\
y_{1110}&y_{0111}&y_{1011}&y_{1101}&y_{0001}&f_{1111}\\
\end{smallmatrix} \right)=1
\end{align*}
\begin{align*}
	\Rank
	\begin{pmatrix}
	f_{1100}f_{1101}&y_{0110}&y_{1010}\\
	y_{0110}&f_{1010}f_{1011}&y_{1100}\\
	y_{1010}&y_{1100}&f_{0110}f_{0111}\\
	\end{pmatrix}=1
\end{align*}
with $\deg f_\bullet=1$.

We describe only the intermediate quotients that are K3 surfaces. 

We find three intermediate K3 surfaces with $15$ nodes, the quotients by $\ker \epsilon_i$, $i=1,2,3$, double planes branched on six lines. Each of them is double covered by a $K3$ with $14$ nodes, the quotient by $\left( \ker \epsilon_i \right) \cap \left( \ker \epsilon_j + \epsilon_k \right)$, $\{i,j,k\}=\{1,2,3\}$ with a symplectic involution by $e_j$. Note that each of these last surfaces is double covered by two further intermediate quotients with $p_g=1$, the quotients by $e_j + e_k$ and $e_j + e_k+e_4$, both giving surfaces with $K$ ample, $K^2=2$ and $8$ nodes. There are special case f the "special Horikawa surfaces" considered in \cite{MR4292206}.
 These pairs of $K3$ are again a specialization of \cite{vGS}*{3.5}, where all plane curves splits as union of lines.

\newpage

\subsection{Family D4}
These surfaces have $K^2=4$.

They are embedded in $\PP(1^8)=\PP^7$ with variables $x_0$, $x_1$, $x_2$, $y_{1100}$, $y_{1010}$, $y_{0110}$, $y_{0001}$, $y_{1111}$ we take the surfaces defined by the equations
\begin{align*}
\Rank
\begin{pmatrix}
f_{1100}(x_j)&y_{0110}&y_{1010}\\
y_{0110}&f_{1010}(x_j)&y_{1100}\\
y_{1010}&y_{1100}&f_{0110}(x_j)\\
\end{pmatrix}&=1&
\begin{aligned}
y_{1111}^2&=f_{1110}(x_j)\\
y_{0001}^2&=f_{1111}(x_j)
\end{aligned}
\end{align*}
with $f_\bullet$ general of respective degrees  $\deg f_{1100}=\deg f_{1010}=\deg f_{0110}=1$  and $\deg f_{1110}=\deg f_{1111}=2$. 

The intermediate quotients that are K3 surfaces form three towers of three K3s corresponding to the chain of subgroups, for $i,j \leq 3$, $i\neq j$
\[
\langle e_i+e_j\rangle \subset \langle e_i+e_j, e_4\rangle \subset \langle e_i, e_j, e_4\rangle 
\]
giving three towers of double covers between K3 surfaces $U_{i,j} \to V_{i,j} \to W_{i,j}$ with respectively $4$, $10$ and $13$ nodes.

\subsection{Family D5}
These surfaces have $K^2=8$. 

They are embedded in $\PP(1^{12})=\PP^{11}$ with variables $x_0,$ $x_1,$ $x_2,$ $y_{11000},$ $y_{10100},$ $y_{01100},$ $y_{00010},$ $y_{00001},$ $y_{11110},$ $y_{11101},$ $y_{00011},$ $y_{11111}$ defined by the equations
\begin{align*}
	\Rank
	\begin{pmatrix}
	f_{11111}(x_j)&y_{00010}&y_{00001}&y_{11111}\\
	y_{00010}&f_{11110}(x_j)&y_{00011}&y_{11101}\\
		y_{00001}&y_{00011}&f_{11101}(x_j)&y_{11110}\\
		y_{11111}&y_{11101}&y_{11110}&f_{11100}(x_j)\\
		\end{pmatrix}&=1
\end{align*}
\begin{align*}
	\Rank
	\begin{pmatrix}
	f_{11000}(x_j)&y_{01100}&y_{10100}\\
	y_{01100}&f_{10100}(x_j)&y_{11000}\\
	y_{10100}&y_{11000}&f_{01100}(x_j)\\
	\end{pmatrix}&=1
	\end{align*}
with $f_\bullet$ general of degree $1$.

There are $48$ intermediate quotients that are $K3$ surfaces, divided in three families, each of them giving several towers of three  consecutive double covers between (four) K3 surfaces. One for each pair $i \neq j$, $i,j =1,2,3$. Namely for each pair of subgroups $H_4 \subset H_8$ with $|H_d|=d$ and
\[ 
\langle e_i+e_j\rangle \subset  H_4 \subset  H_8  \subset \langle e_i, e_j, e_4, e_5\rangle 
\]
we obtain a tower of $4$ K3 surfaces with respectively $8$, $12$, $14$ and $15$ nodes, with the surfaces with $8$ and $15$ nodes depending only on $i$ and $j$.

\section{Burgers}
We recall Laterveer's definition \cite{Robert}*{Definition 3.1}

\begin{definition}
A surface $S$ is called a triple K3 burger if the following conditions are satisfied: 
\begin{enumerate}
\item[(0)] $S$ is minimal, of general type, with $q = 0$ and $p_g = 3$;
\item[(i)] there exist involutions $\sigma_j \colon S \rightarrow S$ ($j = 0, 1, 2$) that commute with one another, and such that the quotients
$\overline{X}_j :=S/\left\langle \sigma_j \right\rangle$ ($j=0,1,2$) are birational to a K3 surface $X_j$;
\item[(ii)] there is an isomorphism
\[
\left( (p_0)^* ,(p_1)^* ,(p_2)^* \right) \colon H^2 (\overline{X}_0,\sO) \oplus H^2 (\overline{X}_1,\sO)\oplus H^2(\overline{X}_2,\sO) \rightarrow H^2 (S,\sO),\]
where $p_j \colon S\rightarrow \overline{X}_j$ denotes the quotient morphism.
\end{enumerate}
\end{definition}

Laterveer's original definition included also the third condition (iii) that the involutions respect the canonical divisor: $\sigma_j^*|K_S|=|K_S|$. We removed that because it is automatic since the pull-back of a canonical divisor by an automorphism  is the divisor of the pull-back of the corresponding differential form. 

Our surfaces not of type A are natural candidates to be triple K3 burger. In fact
\begin{proposition}\label{condition(ii)}
Let $S$ be a smooth $k-$double plane not of type A.

If $S$ is of type B2 set, in the notation of the previous section,  $\sigma_0=e_1+e_2$, $\sigma_1=e_1$ and $\sigma_2=e_2$. Otherwise
set $\sigma_0=e_1+e_2$, $\sigma_1=e_2+e_3$ and $\sigma_2=e_1+e_3$. Then  there is an isomorphism
\[
\left( (p_0)^* ,(p_1)^* ,(p_2)^* \right) \colon H^2 (\overline{X}_0,\sO) \oplus H^2 (\overline{X}_1,\sO)\oplus H^2(\overline{X}_2,\sO) \rightarrow H^2 (S,\sO),\]
where $p_j \colon S\rightarrow \overline{X}_j$ denotes the quotient morphism.
\end{proposition}
\begin{proof}
Let $S$ be a smooth $k-$double plane of type C. So we are considering now the families C3, C4, D3, D4, D5 and E3, and not considering  the family B2 yet.

We know that $H^0(S,K_S)^\chi=0$ unless $\chi=\epsilon_1$, $\epsilon_2$, or $\epsilon_3$. More precisely
\[
H^0(S,K_S) = p_0^* H^0(\overline{X}_0,K_{\overline{X}_0}) \oplus p_1^*H^0(\overline{X}_1,K_{\overline{X}_1}) \oplus p_2^*H^0(\overline{X}_2,K_{\overline{X}_2}) = \CC \oplus \CC \oplus \CC.
\]
which implies the stated isomorphism by the standard Serre duality.

If $S$ is of type B2 the proof follows by the same argument replacing $\epsilon_3$ with $\epsilon_1+\epsilon_2$.
\end{proof}

The following consequence has already been proved by Laterveer for the surfaces of type B2 in \cite{Robert}*{Remark 3.4}
\begin{corollary}\label{whichareK3}
The families B2,  C3,  D3, D4, D5 and E3 are families of triple K3 burgers.
\end{corollary}
\begin{proof}
All our surfaces have ample canonical class, so condition (0) is automatic. 

In Proposition \ref{condition(ii)} we have chosen involutions $\sigma_j$ in each case and proved condition $2$ for them. About condition (i), we have shown that the surfaces $\overline{X}_j$ are nodal K3 surfaces in the previous section.  
\end{proof}

We note that the surfaces in the family C4 are not triple K3 burgers since the quotients $\overline{X}_j$ are three surfaces of general type, more precisely surfaces with $K$ ample, $K^2=2$, $p_g=1$, $q=0$ and $8$ nodes. However each of them is a double cover of a K3 surface with $14$ nodes.


\

\begin{acknowledgements}
	We thank Alice Garbagnati and Robert Laterveer for enlightening discussions via email and for suggesting that we study the K3 quotients of these surfaces. We also thank Alice Garbagnati for pointing a beautiful paper by van Geemen and Sarti to us.\\
	The authors were partially supported by INdAM (GNSAGA).
	The second author was partially supported also by  MIUR PRIN 2017 “Moduli Theory and Birational Classification”.
\end{acknowledgements}



\begin{thebibliography}{99}
	\bibitem{BvBP} B\"{o}hning, C., Graf von Bothmer, H.-C., Pignatelli, R,:
	A rigid, not infinitesimally rigid surface with $K$ ample.
	Boll. Unione Mat. Ital. \textbf{15}, 57--85 (2022); 
	\doi{10.1007/s40574-021-00296-3}
	
	\bibitem{beauville} Beauville, A.:
	L'application canonique pour les surfaces de type g\'{e}n\'{e}ral.
	Invent. Math.\textbf{55}, 121--140 (1979); 
	\doi{10.1007/BF01390086}
	
	\bibitem{Bin} Bin, N.:
	Some examples of algebraic surfaces with canonical map of degree 20.
	C. R. Math. Acad. Sci. Paris \textbf{359}, 1145--1153 (2021); 
	\doi{10.5802/crmath.267}
	
	\bibitem{rigidity} Bauer, I., Pignatelli, R.: Rigid but not infinitesimally rigid compact complex manifolds, Duke Math. J. \textbf{170}, 1757--1780, (2021); \doi{10.1215/00127094-2020-0062}
	
	\bibitem{bidouble} Catanese, F.: Singular bidouble covers and the construction of interesting algebraic surfaces, Amer. Math. Soc., Providence, RI, Contemp. Math., \textbf{241}, 97--120, (1999); \doi{10.1090/conm/241/03630}
	
	\bibitem{Cetraro} Catanese, F.; Differentiable and deformation type of algebraic surfaces, real and symplectic structures, Springer, Berlin, Lecture Notes in Math., \textbf{1938}, 55--167; \doi{10.1007/978-3-540-78279-7\_2}
	
	\bibitem{DuGao} Du, R., Gao, Y.: Canonical maps of surfaces defined by abelian covers, Asian J. Math., \textbf{18}, 219--228, (2014); \doi{10.4310/AJM.2014.v18.n2.a2}
	
	\bibitem{fede2} Fallucca, F.: Examples of surfaces with canonical map of degree 12, 13, 15, 16 and 18, to appear on Ann. Mat. Pura Appl. (4), (2022); \arxiv{2209.06057}
	
	\bibitem{fede1} Fallucca, F., Gleissner, C.: Some surfaces with canonical map of degree 10, 11 and 14, Math. Nachr. (online first), (2023); \doi{10.1002/mana.202200450}
	
	\bibitem{BarbaraRita} Fantechi, B., Pardini, R.: Automorphisms and moduli spaces of varieties with ample canonical class via deformations of abelian covers, Comm. Algebra \textbf{25},  1413--1441, (1997); \doi{10.1080/00927879708825927}
	
	\bibitem{Alice} Garbagnati, A.: Smooth double covers of K3 surfaces, Ann. Sc. Norm. Super. Pisa Cl. Sci. (5) \textbf{19}, 345--386, (2019), \doi{10.2422/2036-2145.201701\_013}
	
	\bibitem{AliceMatteo} Garbagnati, A., Penegini, M.: Hodge structures of K3 types of bidouble covers of rational surfaces, (2022); \arxiv{2212.11566}
	
	\bibitem{gpr} Gleissner, C., Pignatelli, R, Rito, C.: New surfaces with canonical map of high degree, Commun. Anal. Geom. \textbf{30}, 1809--1821, (2022); \doi{10.4310/CAG.2022.v30.n8.a5}
	
	\bibitem{vGS} van Geemen, B., Sarti, A.: Nikulin involutions on $K3$ surfaces, Math. Z. \textbf{255}, 731--753, (2007); \doi{10.1007/s00209-006-0047-6}
	
	\bibitem{Hor1} Horikawa, E.: Algebraic surfaces of general type with small $c^{2}_{1}.$ I, Ann. of Math. (2) \textbf{104}, 357--387, (1976); \doi{10.2307/1971050}
	
	\bibitem{Hor4} Horikawa, E.: Algebraic surfaces of general type with small $c^{2}_{1}$. IV, Invent. Math. \textbf{50}, 103--128, (1978); \doi{10.1007/BF01390285}
	
	\bibitem{LaiYeung} Lai, C.-J., Yeung, S.-K.: Examples of surfaces with canonical map of maximal degree, Taiwanese J. Math. \textbf{25}, 699--716, (2021); \doi{10.11650/tjm/210105}
	
	\bibitem{Robert} Laterveer, R.: Algebraic cycles and triple $K3$ burgers, Ark. Mat. \textbf{57}, 157--189, (2019); \doi{10.4310/ARKIV.2019.v57.n1.a9}
	
	\bibitem{MR4333030} Laterveer, R.: Zero-cycles on Garbagnati surfaces, Tsukuba J. Math. \textbf{45}, 1--12, (2021), \doi{10.21099/tkbjm/20214501001}
	
	\bibitem{MR4292206} Laterveer, R.: Algebraic cycles and special Horikawa surfaces, Acta Math. Vietnam. \textbf{46}, 483--497, (2021), \doi{10.1007/s40306-021-00421-6}
	
	\bibitem{Liedtke} Liedtke, C.: Singular abelian covers of algebraic surfaces, Manuscripta Math. \textbf{112}, 375--390, (2003), \doi{10.1007/s00229-003-0408-y}
	
	\bibitem{ManettiInventiones} Manetti, M.: On the moduli space of diffeomorphic algebraic surfaces, Invent. Math. \textbf{143}, 29--76, (2001), \doi{10.1007/s002220000101}
	
	\bibitem{surveyMLP} Mendes Lopes, M.,  Pardini, R.: On the Degree of the Canonical Map of a Surface of General Type, The Art of Doing Algebraic Geometry, Birkhäuser, Cham., Trends in Mathematics,  (2023); \doi{10.1007/978-3-031-11938-5\_13}
	
	\bibitem{Rita:paper} Pardini, R.: Abelian covers of algebraic varieties, J. Reine Angew. Math. \textbf{417}, 191--213, (1991); \doi{10.1515/crll.1991.417.191}
	
	\bibitem{Ulf} Persson, U.: Double coverings and surfaces of general type, Lecture Notes in Math., Springer, Berlin,\textbf{687}, 168--195, (1978); 
	
	\bibitem{YPG} Reid, M.: Young person's guide to canonical singularities, Proc. Sympos. Pure Math., Amer. Math. Soc., Providence, RI, \textbf{46}, 345--414, (1987)
	
	\bibitem{carlos1} Rito, C.: A surface with canonical map of degree 24, Internat. J. Math. \textbf{28}, (2017); \doi{10.1142/S0129167X17500410}
	
	\bibitem{carlos2}  Rito, C.: Surfaces with canonical map of maximum degree, J. Algebraic Geom. \textbf{31}, 127--135, (2022)
	
	\bibitem{yeung} Yeung, S.-K.: A surface of maximal canonical degree, Math. Ann. \textbf{368}, 1171--1189, (2017); \doi{10.1007/s00208-016-1450-x}
\end{thebibliography}
\end{document}